\documentclass[12pt]{amsart}

\usepackage{amsthm, amsmath, amssymb, mathtools}
\usepackage{microtype}
\usepackage{overpic, graphicx, color}
\usepackage[dvipsnames]{xcolor}
\usepackage{enumerate}
\usepackage[hidelinks,pagebackref,pdftex]{hyperref}
\usepackage{import}

\AtBeginDocument{%
  \def\MR#1{}
}

\usepackage{marginnote}
\long\def\@savemarbox#1#2{\global\setbox#1\vtop{\hsize\marginparwidth 
  \@parboxrestore\tiny\raggedright #2}}
\renewcommand*{\backref}[1]{}
\renewcommand*{\backrefalt}[4]{
   \ifcase #1
   [No citations.]
   \or [#2]
   \else [#2]
   \fi }

\numberwithin{equation}{section}
\theoremstyle{plain}
\newtheorem{theorem}[equation]{Theorem}
\newtheorem{corollary}[equation]{Corollary}
\newtheorem{lemma}[equation]{Lemma}

\newtheorem{proposition}[equation]{Proposition}

\newtheorem*{namedtheorem}{\theoremname}
\newcommand{\theoremname}{testing}

\theoremstyle{definition}
\newtheorem{definition}[equation]{Definition}
\newtheorem{remark}[equation]{Remark}
\newtheorem{question}[equation]{Question}

\newcommand{\refthm}[1]{Theorem~\ref{Thm:#1}}
\newcommand{\reflem}[1]{Lemma~\ref{Lem:#1}}
\newcommand{\refprop}[1]{Proposition~\ref{Prop:#1}}
\newcommand{\refcor}[1]{Corollary~\ref{Cor:#1}}

\newcommand{\refsec}[1]{Section~\ref{Sec:#1}}

\newcommand{\reffig}[1]{Figure~\ref{Fig:#1}}

\newcommand{\HH}{{\mathbb{H}}}
\newcommand{\RR}{{\mathbb{R}}}

\newcommand{\NN}{{\mathbb{N}}}

\newcommand{\from}{\colon} % As in ``f maps _from_ X _to_ Y''.

\newcommand{\bdy}{\partial}
\renewcommand{\setminus}{\smallsetminus}

\newcommand{\vol}{\operatorname{vol}}
\newcommand{\CV}{\operatorname{CV}}
\newcommand{\vtet}{{v_{\rm tet}}}
\newcommand{\voct}{{v_{\rm oct}}}

\title{Alternating links on surfaces and volume bounds}

\author{Efstratia Kalfagianni}
\author{Jessica S.~Purcell}

\address[]{Department of Mathematics, Michigan State University, East
Lansing, MI, 48824, USA}
\email[]{kalfagia@math.msu.edu}

\address[]{School of Mathematics, Monash University, VIC 3800, Australia }
\email[]{jessica.purcell@monash.edu}

\begin{document}

\begin{abstract}
Weakly generalised alternating knots are knots with an alternating projection onto a closed surface in a compact irreducible 3-manifold, and they share many hyperbolic geometric properties with usual alternating knots. For example, usual alternating knots have volume bounded above and below by the twist number of the alternating diagram due to Lackenby. Howie and Purcell showed that a similar lower bound holds for weakly generalised alternating knots. In this paper, we show that a generalisation of the upper volume bound does not hold, by producing a family of weakly generalised alternating knots in the 3-sphere with fixed twist number but unbounded volumes. As a corollary, generalised alternating knots can have arbitrarily small cusp density, in contrast with usual alternating knots whose cusp densities are bounded away from zero due to Lackenby and Purcell. On the other hand, we show that the twist number of a weakly generalised alternating projection does give two sided linear bounds on volume inside a thickened surface; we state some related open questions.
\end{abstract}

\maketitle

%%%%%%%%%%%%%%%%%%%%%%%%%%%%%%%%%%%%%%%%%%%%%%%%%%%%%%%%%%%%%%%%%
\section{Introduction}\label{Sec:Intro}

Alternating knots have a diagram that alternates on the projection plane $S^2$ in $S^3$. There have been several generalisations of alternating knots to project to other surfaces, such as toroidally alternating links of Adams~\cite{Adams:ToroidallyAlt}, generalised alternating links of Ozawa~\cite{Ozawa1}, and $F$-alternating links of Hayashi~\cite{Hayashi}. In his PhD thesis, Howie introduced a family of knots and links which also have diagrams that alternate on surfaces other than the usual projection plane, but without some of the restrictions of previous work. He called these knots and links \emph{weakly generalised alternating}, to distinguish them from Ozawa's links, and proved that the conditions he requires on the diagram are sufficient to guarantee that such knots are prime, nonsplit, with essential checkerboard surfaces~\cite{Howiethesis}. See also \cite[page~5]{HowiePurcell} for some discussion of the differences between these links and others. 

Recently, Howie and Purcell extended the definition of weakly generalised alternating links to include alternating links on surfaces embedded in any compact orientable 3-manifold, not just $S^3$~\cite{HowiePurcell}. They showed that such links have many geometric properties in common with usual alternating links in $S^3$. For example, Menasco was the first to give conditions on usual alternating links that guarantee they are hyperbolic~\cite{Menasco:Alternating}. These results were extended by several authors, such as by Adams~\cite{Adams:ToroidallyAlt}, Hayashi~\cite{Hayashi}, and Ozawa~\cite{Ozawa1} for alternating links on surfaces other than $S^2$ in $S^3$, and by Adams \emph{et al}~\cite{Adamsetal:hyperbolicity}, and Champanerkar, Kofman, and Purcell~\cite{CKP:Biperiodic} for alternating links in $T^2\times\RR$. In \cite{HowiePurcell}, Howie and Purcell give broad conditions that guarantee hyperbolicity in many additional cases, subsuming many of the previous results. They also consider geometry of surfaces embedded in such links, and link volumes.

In this paper, we focus again on volumes. The volume of a usual alternating link on $S^2$ in $S^3$ is bounded above and below by linear functions of the twist number of the link, due initially to work of Lackenby~\cite{lackenby:alt-volume}, with the lower bound improved by Agol, Storm, and W.~Thurston~\cite{AST}, and the upper bound improved by Agol and D.~Thurston~\cite[Appendix]{lackenby:alt-volume}. Note that the upper bound applies to any non-split link in $S^3$, not just alternating.

The problem of determining hyperbolic volumes of link complements has a long history.
In addition to bounds on volumes of alternating links, there are known bounds or improved bounds on volumes of other families, such as highly twisted links \cite{fkp:filling, Purcell:VolHighlyTwisted}, 2-bridge links \cite{GueritaudFuter:2bridge}, links with certain symmetries \cite{FKP:Mrl}, links with high amounts of generalised twisting \cite{Purcell:multiplytwisted, fkp:coils, Adams:GenAugAlt}, chain links \cite{KPR:VolChainLinks}, semi-adequate links \cite{fkp:guts}, weaving links \cite{CKP:Weaving}. Considering upper bounds only, there are better upper bounds on volumes of twisted torus links \cite{CFKNP}, asymptotically sharp upper bounds in terms of crossing number for certain families \cite{CKP:Geommax}, and bounds in terms of triple crossing number for others \cite{Adams:Triple}. There are also improved upper bounds on volumes and simplicial volumes for links in $S^3$ with diagrams with different features \cite{Adams:Bipyramids, DT:RefinedVol, DT:Simplicial}. There remain many open questions on how the volume of a link behaves with respect to its diagrams and other invariants. A goal of research in knot and link volumes is to find relations of such behaviour for broader families of knots, or to characterise properties of diagrams that make such relations possible.

Weakly generalised alternating links form a much broader family of links than merely alternating links, but they have similarities with alternating links that allow extension of some of the tools effectively used in the usual alternating case. Thus it is a natural question to investigate their volumes.

In \cite{HowiePurcell}, Howie and Purcell show that a lower bound on volume, similar to that for alternating links, holds for a weakly generalised alternating link satisfying mild hypotheses. That is, they define the twist number on the projection surface and show there is a linear lower bound on volume in terms of the twist number. However, they do not discuss the question of the upper bound, motivating the following.

\begin{question}\label{Ques:UpperBound}
Let $K$ be a link with a weakly generalised alternating projection on an embedded surface $F$ in a compact, irreducible 3-manifold $Y$. Is there a linear function of the twist number on the projection surface $F$ that gives an upper bound on the volume of the link complement $\vol(Y-K)$? 
\end{question}

In this paper, we address this question from different directions.
Our first result is to show that the answer is no when $Y=S^3$. Thus no linear function of twist number gives an upper bound on volumes for all weakly generalised alternating links.

\begin{theorem}\label{Thm:MainVolumeIntro}
  There exists a family of weakly generalised alternating knots $\{K_n\}_{n\in\NN}$ in $S^3$ with the following properties:
  \begin{itemize}
  \item The projection surface $F$ of each $K_n$ is a Heegaard surface of genus two.
  \item The volume $\vol(K_n)\to \infty$ as $\to\infty$. 
  \item The twist number of the diagram on $F$ is constant for all $n$.
  \end{itemize}
\end{theorem}
  Theorem~\ref{Thm:MainVolumeIntro} is an immediate consequence of a more general theorem stated as \refthm{MainVolume} below. 

One application of \refthm{MainVolumeIntro} is to cusp densities. Lackenby and Purcell~\cite{LPalte} showed that usual alternating knots on $S^2$ in $S^3$ have cusp volumes bounded below in terms of the twist number of the knot, and that cusp densities of these knots are universally bounded away from zero.
Recently, Bavier~\cite{Bavier} generalised the lower bound on cusp volume to weakly generalised alternating knots satisfying mild hypotheses.
However, \refthm{MainVolumeIntro} shows that a similar cusp density result cannot hold. Combining \refthm{MainVolumeIntro} with a result of Burton and Kalfagianni~\cite{BuKa} gives the following.

\begin{corollary}\label{Cor:Cusps}
There exist weakly generalised alternating knots with arbitrarily small cusp density.
\end{corollary}

The construction of the links in \refthm{MainVolumeIntro}, or more precisely \refthm{MainVolume}, uses heavily the fact that the links lie on a Heegaard surface $F$ of genus two in $S^3$. Recall that such a surface bounds a handlebody of genus two on both sides. Therefore $F$ is compressible, with several compressing discs, allowing us to twist the diagram without affecting the projection surface or the ambient manifold $S^3$. It is interesting to ask whether there exist linear upper bounds on volume in terms of twist number $t_F(\pi(K))$ in the case that the projection surface $F$ is incompressible in $Y$. In  fact, we show that there \emph{is} an upper bound when $Y=S\times [-1,1]$ and $K$ has a weakly generalised alternating projection onto $F=S\times\{0\}$. More specifically, we show the following, where the lower bound is due to Howie and Purcell \cite{HowiePurcell}.

\begin{theorem}\label{Thm:SxIUpperBound}
Let $S$ be a closed orientable surface of genus at least one, and let $K$ be a link that admits a twist-reduced weakly generalised alternating projection onto $F=S\times\{0\}$ in $Y=S\times[-1,1]$, for which the complementary regions of $F-\pi(K)$ are discs. Then the interior of $Y-K$ admits a hyperbolic structure. If $S=T^2$, then we have
\[  {\voct\over 2}\ t_F(\pi(K)) \leq \vol(Y-K) < 10\,\vtet\cdot t_F(\pi(K)), \]
where $\vtet = 1.01494\dots$ is the volume of a regular ideal tetrahedron and $\voct = 3.66386\dots$ is the volume of a regular ideal octahedron.

If $S$ has genus at least two, 
\[  {\voct\over 2}\ ( t_F(\pi(K))- 3\chi(F)) \leq \vol(Y-K) < 6\,\voct\cdot t_F(\pi(K)).\]
\end{theorem}

In \refthm{SxIUpperBound}, if $S$ is a torus then there is a unique hyperbolic structure on the interior of $Y-K$, and the volume is the hyperbolic volume of this structure. However, if $S$ has genus at least two, there are many hyperbolic structures. By $\vol(Y-K)$ we mean the volume of the unique hyperbolic structure on the interior of $Y-K$ for which the boundary components corresponding to $S\times\{-1\}$ and $S\times\{1\}$ are totally geodesic. This agrees with the definitions in \cite{AdamsCalderonMayer, HowiePurcell}.

Theorem~\ref{Thm:SxIUpperBound} has a somewhat surprising corollary, namely that there does exist an upper bound on the volume of a weakly generalised alternating link on a Heegaard torus $F$ in $S^3$; this is \refcor{TorAltUpperBound}. (Recall that a Heegaard torus bounds a solid torus on each side.)
Thus the fact that the family in \refthm{MainVolume} lies on Heegaard surfaces of genus at least two is necessary.

Theorem~\ref{Thm:SxIUpperBound} should be compared to related work. Adams, Calderon, and Mayer~\cite{AdamsCalderonMayer} found upper bounds on volumes of an alternating knot $K$ on $F=S\times\{0\}$ in $Y=S\times[-1,1]$ in terms of the crossing number $c_F(\pi(K))$. For $F=T^2\times \{0\}$ in $Y=T^2\times\RR$, Champanerkar, Kofman and Purcell~\cite{CKP:Biperiodic} found sharp upper bounds on volume in terms of the complementary regions in $F-\pi(K)$. 

By contrasting \refthm{MainVolume} and \refthm{SxIUpperBound}, we are led to the following.

\begin{question}
  For which manifolds $Y$ and which projection surfaces $F$ will a link with a weakly generalised alternating projection to $F$ in $Y$ admit a linear upper bound on volume, or simplicial volume, in terms of twist number $t_F(\pi(K))$? Will such a bound hold if $F$ is incompressible?
\end{question}

\noindent{\bf{Acknowledgements.}}
We thank Josh Howie for pointing out that \refcor{TorAltUpperBound} extends to lens spaces and toroidally alternating links. 
This work is based on research done while  Kalfagianni was on sabbatical leave, supported by NSF grants DMS-1708249, DMS-2004155 and a grant from the Institute for Advanced Study (IAS) School of Mathematics. The authors thank the IAS for helping to support Purcell's visit during this time. Purcell was also supported by the Australian Research Council, grant DP210103136.

%%%%%%%%%%%%%%%%%%%%%%%%%%%%%%%%%%%%%%%%%%%%%%%%%%%%%%%%%%%%%%%%%
\section{Definitions}\label{Sec:Definitions}

\subsection{Weakly generalised alternating links}
Throughout, we will take our projection surface $F$ to be an oriented, closed, connected surface embedded in a compact, irreducible, oriented 3-manifold $Y$, possibly with boundary. For the first part of the paper, namely through \refsec{Cusps}, $Y$ will be $S^3$. In \refsec{Thickened}, $Y=S\times[-1,1]$ and $F=S\times\{0\}$. (Disconnected projection surfaces appear in \cite{Howiethesis, HowiePurcell}, but we will not consider that generality here.) Given a link $K \subset F\times [0, \ 1] \subset Y$ in general position, the image of $K$ under a projection $\pi\from F\times [0, \ 1] \to F$, along with over--under crossing information at double points, is called a generalised link diagram, and is denoted by $\pi(K)$.
A generalised link diagram is said to be alternating if it can be oriented such that in each region of $F-\pi(K)$, crossings alternate over and under. 

\begin{definition}\label{Def:Wprime}
A generalised diagram $\pi(K)$ on $F$ is \emph{weakly prime} if, for any embedded disc $D\subset F$ that intersects $\pi(K)$ exactly twice, we have:
\begin{itemize}
\item If $F\neq S^2$, then $\pi(K)\cap D$ is an arc without crossings.
\item If  $F=S^2$, then $\pi(K)$ has at least two crossings and either $D$ or the complement of $D$ intersects $\pi(K)$ in an arc without crossings.
\end{itemize}
The diagram $\pi(K)$ is called \emph{reduced alternating} if it is alternating, weakly prime and each component of $K$ projects with at least one crossing on $F$.
\end{definition}

Howie has shown that the conditions on a reduced alternating diagram in $S^3$ suffice to show that the link is nontrivial, nonsplit, and prime \cite{Howiethesis}.

\begin{definition}
We say that the \emph{representativity} of a diagram $\pi(K)\subset F$ is at least $\rho$, and we write $r(\pi(K), F)\geq \rho$, if the boundary of every compressing disc of $F$ intersects $\pi(K)$ in at least $\rho$ points. If $F$ has no compressing disc, define $r(\pi(K), F) = \infty$.

We say that the \emph{edge-representativity} of $\pi(K)$ is at least $\rho$, and we write $e(\pi(K), F)\geq \rho$, if every essential curve on $F$ intersects $\pi(K)$ at least $\rho$ times.
\end{definition}

The terms representativity and edge-representativity for knots and links were introduced by Howie~\cite{Howiethesis}, motivated by terminology in topological graph theory; see for example the survey article~\cite{KawaMohar}.
Observe that the edge-representativity of a knot or link will always be at most its representativity. 

The following definition originally appeared in Howie's thesis~\cite{Howiethesis} for links in $S^3$, and was generalised to compact irreducible 3-manifolds by Howie and Purcell~\cite{HowiePurcell}.
\begin{definition}
A diagram $\pi(K)\subset F\subset Y$ is called \emph{weakly generalised alternating} (WGA), if it is reduced alternating and:
\begin{itemize}
\item the regions $F\setminus \pi(K)$ admit a checkerboard coloring, and
\item the representativity satisfies $r(\pi(K), F)\geq 4$.
\end{itemize}
A link $K$ that admits a WGA diagram is called \emph{weakly generalised alternating} (WGA).
\end{definition}

Howie and Purcell showed that the complement of a weakly generalised alternating link is nontrivial, irreducible, and boundary irreducible, again implying that the link is nonsplit in $S^3$ \cite[Corollary~3.16]{HowiePurcell}.

\begin{definition}\label{Def:TwistNumber}
Let $\pi(K)$ be a WGA projection on a surface $F$. A \emph{twist region} of $\pi(K)$ on $F$ is a portion of the underlying diagram graph of $\pi(K)$ that is either:
\begin{itemize}
\item a string of bigon regions arranged vertex to vertex that is maximal (i.e.\ not contained in a larger string of bigons of $\pi(K)$),
\item or a single vertex (i.e.\ crossing) adjacent to no bigons. 
\end{itemize}

The \emph{twist number} of $\pi(K)$ on $F$, denoted by $t_F(\pi(K))$ is the number of twist regions of $\pi(K)$.
\end{definition}

\begin{definition}
\label{Def:TwistReduced} 
A WGA diagram $\pi(K)\subset F$ is called \emph{twist reduced} if whenever there is a disc $D\subset F$ such that $\bdy D$ intersects $\pi(K)$ exactly four times adjacent to two crossings, then one of the following holds:
\begin{itemize}
\item $D$ contains a (possibly empty) sequence of bigons that is part of a larger twist region containing the two crossings, or
\item $F\setminus D$ contains a disc $D'$, with $\partial D'$ intersecting $\pi(K)$ four times adjacent to the same two crossings as $\partial D$, 
and $D'$ contains a string of bigons that forms a larger twist region containing the original two crossings. 
See \reffig{TwistReduced}.
\end{itemize}
\end{definition}

\begin{figure}
  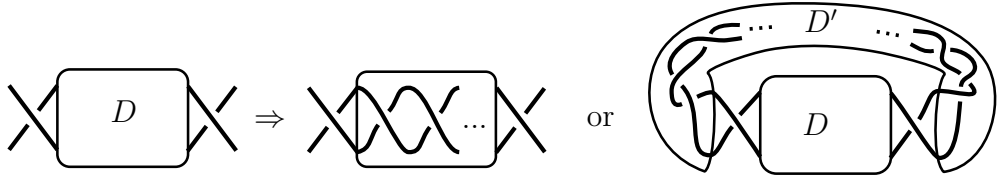
  \caption{A twist reduced diagram. Figure modified from \cite{HowiePurcell}.}
  \label{Fig:TwistReduced}
\end{figure}

Note that any WGA diagram can be modified to be twist reduced by performing a flype on a disc that does not contain bigons, moving two twist regions together. See also the discussion after Definitions~6.3 and~6.4 in \cite{HowiePurcell}. 
Observe also that in the case that $F=S^2$, our definition of twist reduced agrees with that of Lackenby for usual alternating links~\cite{lackenby:alt-volume}; compare \reffig{TwistReduced} with Figure~3 of that paper.

The following is a special case of a result of Howie and Purcell~\cite[Theorem~1.1]{HowiePurcell}.

\begin{theorem}[Howie--Purcell]\label{Thm:HP}
Let $\pi(K)$ be a weakly generalised alternating projection of a link $K$ on a closed surface $F$ of genus at least one in the 3-manifold $Y=S^3$, or $Y=S\times[-1,1]$ and $F=S\times\{0\}$ for a surface $S$ homeomorphic to $F$. Suppose that $Y-F$ is atoroidal and that all the regions
of $F-\pi(K)$ are discs. Finally, suppose that $\pi(K)$ is twist-reduced and $r(\pi(K), F)>4$. Then the following hold.
\begin{itemize}
\item $Y-K$ is hyperbolic.
\item The checkerboard surfaces of $\pi(K)$ are essential in $Y-K$.
\item The hyperbolic volume satisfies
\[ \vol(Y-K) \geq {\voct\over 2}\ ( t_F(\pi(K))-\chi(F)-\chi(\bdy Y)).\]
\end{itemize}
\end{theorem}

In the theorem, the \emph{checkerboard surfaces} of $\pi(K)$ are obtained by taking the checkerboard colouring of $F-\pi(K)$ into white and shaded regions. Under the above hypothoses, these regions are coloured discs. Complete the union of the white discs into a spanning surface by joining two such dics by a twisted band when they are adjacent across a crossing of $\pi(L)$. Similarly for shaded discs. The two surfaces that arise are the checkerboard surfaces for $\pi(L)$. 

We finish this subsection with the following result, which is a special case of results from \cite{HowiePurcell} that we will need in later sections. 

\begin{theorem}[Howie--Purcell]\label{Thm:HP2}
Let $\pi(K)$ be a weakly generalised alternating projection of a knot $K$ on a closed surface $F$ of genus at least one in $S^3$. Suppose that $S^3-F$ is atoroidal and $r(\pi(K), F)>4.$ Then the following are true:
\begin{enumerate}
\item Suppose that $T$ is an essential torus in $S^3-K$. Then the checkerboard surfaces of $\pi(K)$ cut $T$ into annuli such that each boundary component of each of the annuli lies entirely in a single region of $F-\pi(K)$.
\item Suppose that $S^3-K$ contains an essential annulus. Then $\pi(K)$ has only one twist region.
\end{enumerate}
\end{theorem}
\begin{proof}
Part (1) is a special case of Proposition~4.10 of \cite{HowiePurcell}. We note that if  $S^3-K$ contains an essential torus, then $F-\pi(K)$ must contain some regions that are not discs. 
Part (2) is a special case of Theorem~4.6 of \cite{HowiePurcell}. 
\end{proof}

%%%%%%%%%%%%%%%%%%%%%%%%%%%%%%%%%%%%%%%%%%%%%%%%%%%%%%%%%%%%%%%%%

\subsection{Cusped manifolds}
We will use the existence and non-existence of upper volume bounds to draw conclusions on cusp density of WGA knots. In this section, we review the necessary terminology. 

Suppose $\overline{M}$ is a compact orientable 3-manifold with one component of $\bdy\overline{M}$ a single torus, for example a knot complement, and suppose the interior $M \subset \overline{M}$ admits a complete hyperbolic structure. We say $M$ is a \emph{cusped manifold}. One end of $M$ is homeomorphic to $T^2\times [1,\infty)$. Denote this by $C$. Under the covering map $\rho\from \HH^3\to M$, this end is geometrically realised as the image of a horoball $H_i\subset \HH^3$, and the preimage $\rho^{-1}(\rho(H_i))$ is a collection of horoballs. By shrinking $H_i$ if necessary, we may ensure that these horoballs have disjoint interiors in $\HH^3$. For such a choice of $H_i$, $\rho(H_i) = C$ is said to be a \emph{horoball neighbourhood} of the \emph{cusp} $C$, or \emph{horocusp} in $M$.

For a cusped manifold $M$, a 1-parameter family of horoball neighbourhoods, parameterised by size of the neighbourhood, is obtained by expanding the horoball $H_i$ while keeping the same limiting point on the sphere at infinity. Taking the preimage in $\HH^3$, expanding $H_i$ expands all horoballs in the collection $\rho^{-1}(C)$. We may expand until the collection of horoballs $\rho^{-1}(\cup C)$ become tangent, and cannot be expanded further while keeping their interiors disjoint. This is a \emph{maximal cusp}. For the case here, in which $M$ has a single cusp, there is a unique choice of expansion, giving a unique maximal cusp.

\begin{definition}
For a hyperbolic knot $K$, with maximal cusp $C$, the \emph{cusp volume} of $K$, denoted by ${\rm CV}(K)$, is the volume of $C$. This volume is half of the Euclidean area of $\partial C$; see, for example, \cite[Exercise~2.31]{Purcell:HypKnotTheory}.
\end{definition}

The following is a result of Burton and Kalfagianni on cusp volumes,~\cite[Theorem~1.1]{BuKa}.

\begin{theorem}[Burton--Kalfagianni]\label{Thm:BK}
Let $K$ be a hyperbolic knot in $S^3$ with maximal cusp $C$ and cusp volume $\CV(K)$. Suppose that $S_1$ and $S_2$ are essential spanning surfaces in $S^3-K$ and let $i(\partial S_1, \partial S_2)\neq 0$ denote the minimal intersection number of $\partial S_1, \partial S_2$ on the torus $\partial C$. 
Then
\[ \CV(K) \leq 9 \ 
 \dfrac{(|\chi(S_1)| + |\chi(S_2)|)^2}{i(\partial S_1, \partial S_2)}.\]
\end{theorem}

%%%%%%%%%%%%%%%%%%%%%%%%%%%%%%%%%%%%%%%%%%%%%%%%%%%%%%%%%%%%%%%%%
\section{The construction of the knots}\label{Sec:Construction}

In this section, we will construct sequences of weakly generalised alternating knots in $S^3$ that satisfy \refthm{MainVolume}. In fact we will construct two distinct families of such knots.
In the first family, which we will denote
$\{K_{n}\}_{n\in \NN}$, some components $F-\pi(K_n)$ will not be discs while in the second family, denoted $\{K^{*}_{n}\}_{n\in \NN}$, all the complementary regions of the projections $\pi(K^{*}_n)$ will be discs.
For WGA knots in $S^3$ with the latter property, \refthm{HP} states that they are hyperbolic and their volume is bounded below by the quantity $t_F(\pi(K^{*}_n))-\chi(F)$.
For the knots $\{K_{n}\}_{n\in \NN}$, we will use tools from \cite{HowiePurcell} to show that their complements are hyperbolic; the requirement that all the regions of the WGA knot projections are discs is not necessary for the hyperbolicity claim of \cite[Theorem~1.1]{HowiePurcell}.

\subsection{The construction}\label{Sec:SubConstruction}
Begin with a generalised projection surface that is a genus two surface $F$ embedded as a standard Heegaard surface in $S^3$, sketched as on the left of \reffig{Template}.
We construct two graphs on $F$, which we will call the \emph{templates} for the knots we construct.

The first template is obtained by running six parallel curves around each of the handles of the bounded handlebody cut off by $F$ in $S^3$. The curves lie in the shaded region shown in the middle of \reffig{Template}. Three additional parallel curves encircle a disc in $F$, and each of these three curves intersects each the six strands exactly four times. Together, these six plus three curves form a graph on $F$, with 72 vertices in a grid-like pattern. This is shown in the middle of \reffig{Template}. Notice that this graph can be embedded on the standard plane of projection in $S^3$. Alternatively, we view the standard plane of projection in $S^3$ as cutting through the Heegaard surface $F$, splitting it into a 3-punctured sphere above and a 3-punctured sphere below. This first template lies entirely on the 3-punctured sphere above the plane of projection.

\begin{figure}
  \includegraphics{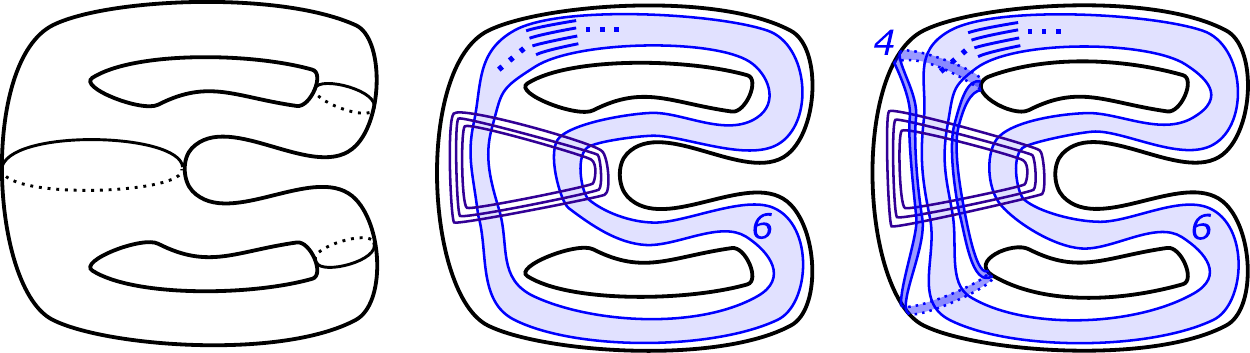}
  \caption{Left: a genus-2 Heegaard surface for $S^3$. Middle: the first template, giving a graph with 72 vertices and complementary regions all discs, except a single 3-punctured sphere region across the back. Right: the second template, giving a graph with 120 vertices and complementary regions all discs.}
  \label{Fig:Template}
\end{figure}

For the second template, add four additional curves to the first template. These curves are again parallel, disjoint from the six parallel strands running around the handles, and intersect each of the three other curves four times. Moreover, these curves run below the standard plane of projection for $S^3$, to run across the lower 3-punctured sphere of $F$ in eight arcs, in two parallel pairs. This is shown on the right of \reffig{Template}. Each of the four curves adds twelve vertices to the graph on $F$.

Make each template into an alternating link by assigning over/under crossing information to each vertex in an alternating fashion. Denote the link arising from the first template by $\hat{K}$, and the link arising from the secton by $\hat{K}^*$, with diagrams $\pi(\hat{K})$ and $\pi(\hat{K}^*)$, respectively.

\begin{lemma}\label{Lem:TemplateCheckTwred}
  The diagrams $\pi(\hat{K})$ and $\pi(\hat{K}^*)$ are checkerboard coloured and twist reduced.
\end{lemma}

\begin{proof}
The best way to see both results is by inspection. The two diagrams are explicit. It is a somewhat soothing exercise to sketch the diagrams and colour the checkerboard surfaces by hand.

Alternatively, $\pi(\hat{K})$ can be projected onto the usual plane of projection in $S^3$, giving an ordinary alternating link, which must be checkerboard coloured. Note that in this ordinary alternating link, the region on the outside has the same colouring as the two regions inside the loops that ran around the handles before the projection. Therefore the colouring transfers to the genus-2 surface, giving a consistent colouring on the single 3-punctured sphere region under the plane of projection. Adding four parallel curves to this, as in $\pi(\hat{K}^*)$, does not affect most faces. One can check, using the fact that there are exactly four new strands, that the faces that are changed can have colouring modified to give a unified checkerboard colouring of the diagram.

As for twist reduced, suppose there is a disc $D$ in either $\pi(\hat{K})$ or $\pi(\hat{K}^*)$ with $\bdy D$ meeting the diagram exactly four times, adajcent to exactly two crossings, as in \reffig{TwistReduced}, left. Isotope $\bdy D$ to run through the adjacent crossings. This splits $\bdy D$ into two arcs between crossings, with each such arc lying in a single region of the diagram. Moreover, these two regions have the same colour. Thus such a disc will correspond to two regions of the same colour meeting at two distinct crossings. By inspection, most regions of the diagram are quads with four distinct regions meeting the four vertices (crossings) of the quad, so such a disc cannot run there. In the first template there is the 3-holed sphere region, but all similar coloured regions adjacent to the 3-holed sphere across crossings are distinct. Similarly in the second template there is another large disc region, but it meets only distinct regions at each of its crossings. Finally, in both templates there are exactly two bigon faces, and thus a disc meeting exactly two crossings encircling the bigon. This satisfies the definition of twist reduced.
\end{proof}

Next, replace each single crossing on $F$ in the link diagram $\pi(\hat{K})$ and $\pi(\hat{K}^*)$ with an alternating twist region with a high number of crossings so that the link remains alternating on $F$. Moreover, choose the twist region such that if the parallel strands of the template are oriented in the same direction, then the replacement preserves the orientation, as shown in \reffig{TwistOrientation}. Under this replacement, many bigon faces are added.
Choose the parity of the number of crossings in each twist region so that the result is a diagram of a knot in both cases, rather than a link. 

\begin{figure}
\includegraphics{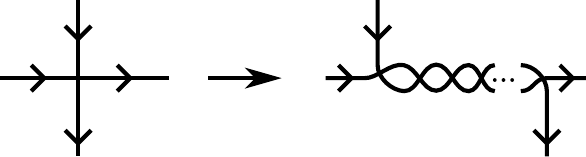}
  \caption{Insert twist regions at each crossing to preserve an orientation.}
  \label{Fig:TwistOrientation}
\end{figure}

Denote the knot arising from the first template by $K$, and the knot arising from the second by $K^{*}$.

\begin{lemma}\label{Lem:WeaklyPrime}
  The diagrams of the knots $K$ and $K^{*}$ are checkerboard coloured, twist reduced, and weakly prime. 
\end{lemma}

\begin{proof}
The diagrams of $\pi(\hat{K})$ and $\pi(\hat{K}^*)$ are checkerboard coloured and twist reduced by \reflem{TemplateCheckTwred}. When adding bigons to these diagrams to form $\pi(K)$ and $\pi(K^*)$, the diagram remains checkerboard coloured and twist reduced, because new bigons are added only to existing twist regions.

Showing weakly prime requires a little more work. 
Consider the three closed curves of intersection of $F$ with the usual plane of projection in $S^3$; call them $x$, $y$, $z$. These divide $F$ into two pairs of pants, or 3-punctured spheres, one above the plane of projection and one below. The entire knot $K$ lies above the plane of projection, and is disjoint from the 3-punctured sphere below the plane of projection. Most of the knot $K^{*}$ lies above the plane of projection, but eight unknotted arcs of the diagram lie below. 

Now, suppose there is an embedded disc $D$ on $F$ whose boundary $\bdy D$ intersects the diagram $\pi(K)$ or $\pi(K^{*})$ in exactly two points. The boundary $\bdy D$ must intersect each of the curves $x$, $y$, and $z$ in an even number of intersection points. If it intersects any of $x$, $y$, $z$ in more than zero points, then there exists an outermost arc of $\bdy D$ co-bounding a disc with an arc of $x$, $y$, or $z$, forming a bigon $B$. 

Suppose first that the boundary of the bigon $\bdy B$ is disjoint from the knot diagram. Then we may use $B$ to isotope $D$ through $x$, $y$, or $z$, to reduce the number of intersections with $x$, $y$, and $z$. Repeat this a finite number of times to remove all bigons that are disjoint from the diagram. 

Next suppose that the boundary of $B$ intersects the diagram. Again $B$ is a disc, so $\bdy B$ intersects the knot diagram twice. Note that $B$ is embedded either above or below the plane of projection. If below, there are no crossings that $B$ could enclose, hence it bounds a single unknotted arc. If above, $B$ is a disc in a 3-punctured sphere containing a graph whose only vertices come from bigons replacing vertices in a grid of 72 or 120 points, with edges running from the outside of this grid. A grid is well known to form a prime diagram, with primeness unchanged by replacing a vertex with a twist region, so $B$ bounds a single unknotted arc of the diagram.

The boundary $\bdy B$ is made up of two arcs, one on $\bdy D$ and one on $x$, $y$, or $z$. The diagram either intersects $\bdy D$ twice, or $\bdy D$ once and $x$, $y$, or $z$ once (if the diagram is that of $K^{*}$). In the first case, the diagram intersects $\bdy D$ twice in $\bdy B$; these must be the only intersections of $\bdy D$ with the diagram. Hence since $B$ bounds an unknotted arc, so does $D$. In the second case, use $B$ to isotope $D$ along the unknotted arc of the diagram, pushing it past $x$, $y$, or $z$, and removing two intersections with these curves. After a finite number of steps, there are no bigons $B$.

Then $D$ lies either completely above or completely below the plane of projection. As before, it follows that it meets the diagram in a single unknotted arc. Thus the diagram is weakly prime. 
\end{proof}

Let $u$, $v$, and $w$ be the curves shown in \reffig{RedCurves}, which bound meridian discs in one of the handlebodies. Note they cut the surface $F$ into two 3-holed spheres. 
We may and will choose these curves to be disjoint from the vertices of our templates where the crossings of $\pi(K)$ or $\pi(K^{*})$ occur.

\begin{figure}
  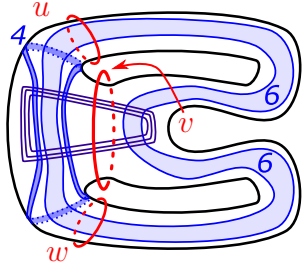
  \caption{Shown are the curves $u$, $v$, and $w$. }
  \label{Fig:RedCurves}
\end{figure}

\begin{lemma}\label{Lem:K*EdgeRep}
  The edge-representativity $e(\pi(K^{*}), F)$ is equal to four.
\end{lemma}

\begin{proof}
The nontrivial curve $x$ on $F$ intersects $\pi(K^{*})$ four times, so edge-representativity is at most four. To show it is equal to four, we need to show that any other essential curve on $F$ intersects $\pi(K^{*})$ at least four times.

Let $\ell\subset F$ be an essential curve. Consider how $\ell$ intersects the 3-holed spheres obtained by cutting $F$ along the curves $u$, $v$, and $w$ of \reffig{RedCurves}. If $\ell$ is completely contained in one of the two 3-holed spheres, then the fact that $\ell$ is essential on $F$ means it is parallel to one of the three boundary components. But in either 3-holed sphere, the graph of the template includes six strands of the diagram running between each pair of boundary components, separated by a large grid. A closed curve must intersect this graph at least six times; it follows that $\ell$ intersects $\pi(K^{*})$ at least six times.

Next suppose $\ell$ intersects both 3-holed spheres. Then it must meet each in a collection of arcs. If any of these arcs co-bounds a bigon with one of $u$, $v$, $w$, then $\ell$ can be simplified by isotoping through the bigon. Such an isotopy will at worst, only decrease the number of intersections of $\pi(K^{*})$ with $\ell$, so we assume now that there are no bigons.

Because $\ell$ must alternate running between the two 3-holed spheres, it must have an essential arc on the left side. This is an essential arc on the 3-holed sphere, running from one boundary component to another, possibly the same component. If it runs from one boundary component to a distinct one, it is a \emph{seam}. Otherwise it is a \emph{wave}. Each seam must run through the four parallel strands appearing in $K^{*}$ and not $K$. Each wave either runs through these strands, or through the six parallel strands or the grid portion of the graph appearing in both $K$ and $K^{*}$. Thus in all cases, the arc meets $\pi(K^{*})$ at least four times. It follows that the edge-representativity is four.
\end{proof}

The next step of the construction is to modify the two diagrams of $K$ and $K^{*}$ by performing full (Dehn) twists of the handlebody $H_1$ along the discs bounded by curves $u$, $v$, and $w$, which we denote as \emph{twisting along $u$, $v$, and $w$}, for short. We will twist at least eight times along each curve. This adjusts the diagrams of $K$, $K^{*}$ by adding at least eight full twists, forming new knots, which we denote by $K_T$ and $K_T^{*}$.

Twisting along $u$, $v$, and $w$ gives a homeomorphism of the handlebody $H_1$. It restricts to a homeomorphism of $F$, that is a product of powers of Dehn twist along the curves.
In our case since we will twist  at least eight times along each curve, the power of each Dehn twist will be at least eight.

\begin{lemma}\label{Lem:KTalternating}
Twisting along $u$, $v$, and $w$ yields knots $K_T$ and $K_T^{*}$ with diagrams that are alternating on $F$, checkerboard coloured, twist reduced, and weakly prime.
Further, the edge-representativity of $\pi(K_T^{*})$ is four.
Finally, the complementary regions of the diagram on $F$ have the same topological type as the complementary regions of $\pi(K)$ or $\pi(K^{*})$ on $F$. In particular, there is still one 3-punctured sphere region on $F-\pi(K_T)$, and all other regions on $F-\pi(K_T)$ and $F-\pi(K_T^{*})$ are discs.
\end{lemma}

\begin{proof}
Twisting along any curve $u$, $v$ or $w$ restricts to a power of a Dehn twist on $F$; such a homeomorphism is the identity away from an annulus neighbourhood of the curve. That annulus intersects the knot diagram in six essential arcs running from one boundary component to the other. The arcs subdivide the annulus into six discs that are coloured in a checkerboard fashion. After twisting, the arcs run into the annulus, around its core some number of times, and then out, but no new crossings are added. Because the knot was alternating on $F$ before twisting, the new knot remains alternating on $F$. The checkerboard coloured discs between arcs of the knot diagram are mapped to new discs, with checkerboard colouring preserved under twisting.

To see that the diagram is still weakly prime, suppose a disc $D$ on $F$ meets the diagram $\pi(K_T)$ or $\pi(K_T^{*})$ twice transversally in edges. Apply the inverse Dehn twist to the picture. This preserves $F$ and takes diagram to $\pi(K)$ or $\pi(K^{*})$, respectively, and takes $D$ to a disc now meeting $\pi(K)$ or $\pi(K^{*})$ twice. Because $\pi(K)$ and $\pi(K^{*})$ are weakly prime by \reflem{WeaklyPrime}, the disc must contain no crossings, both before and after homeomorphism, and $\pi(K_T)$ and $\pi(K_T^{*})$ are weakly prime. A similar argument shows the diagram is still twist reduced after twisting.
  
As for edge-representativity, recall that
by Lemma~\ref{Lem:K*EdgeRep} any essential curve on $F$ meets $\pi(K^{*})$ at least four times. Since surface homeomorphisms map essential curves to essential curves and preserve intersection numbers of curves, it follows that any essential curve on $F$ meets $\pi(K_T^{*})$ at least four times. Thus the edge-representativity of $\pi(K_T^{*})$ is four.
\end{proof}

Recall that by construction the crossings of $\pi(K_T)$ or $\pi(K_T^{*})$ are associated with the vertices of the corresponding templates. 
In particular we have arranged so that there are no crossings in the annuli where the Dehn twists along $u$, $v$, or $w$ are supported, and no crossings in the 3-punctured sphere below the plane of projection.
All crossings of $\pi(K_T)$ and $\pi(K_T^{*})$ on $F$ lie in one of two discs on $F$, one on either side of the curve $v$. We call these the \emph{alternating tangles} of $\pi(K_T)$ or $\pi(K_T^{*})$. They are illustrated for the template of $K^{*}$ in \reffig{AltTangle}, left, in the regions shaded in purple. 

\begin{figure}
  \includegraphics{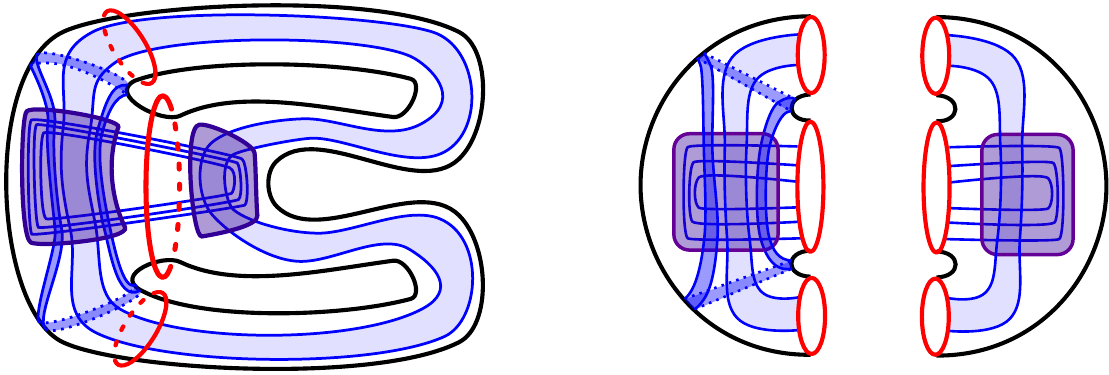}
  \caption{Left: The alternating tangles of $\pi(K_T^{*})$ are shaded in purple; $\pi(K_T)$ is similar. Right: Cutting $H_1$ along discs bounded by $u$, $v$, $w$.}
  \label{Fig:AltTangle}
\end{figure}

\begin{lemma}\label{Lem:Representativity}
Each of the knot diagrams $\pi(K_T)$ and $\pi(K_T^{*})$ has representativity at least six.
\end{lemma}

\begin{proof}
The surface $F$ divides $S^3$ into two handlebodies. Denote the bounded handlebody by $H_1$, and the unbounded one (i.e.\ that contains infinity) by $H_2$. We must argue that every curve $C$ that bounds a compressing disc in either $H_1$ or in $H_2$ will intersect the diagram in at least six points. Observe first that $u$, $v$, and $w$ bound meridian disc of $H_1$ meeting the diagram at least six times. We need to show a similar result for any choice of a curve $C$ bounding a compressing disc $E$. 

Suppose first that $C$ bounds a compressing disc for $H_1$. Note that the curves $u, v, w$ cut the surface $F$ into two pairs of pants, $P_1$, $P_2$, and the discs bounded by $u, v, w$ cut $H_1$ into 3-balls. This is shown on the right of \reffig{AltTangle}.
Isotope $C$ so that its intersection with $u\cup v\cup w$ is minimal on $F$. By further isotopy we can arrange so that the disc $E$ bounded by $C$ intersects those bounded by $u, v, w$ in arcs (not simple closed curves).
Now an arc of this intersection that is innermost on $E$ will define a wave on one of $P_1,P_2$. That is, it forms an essential arc on $P_i$ with both of its boundary points on the same component of $\partial P_i$. 
Such a wave will separate the two boundary components of $P_i$  that is disjoint from. Thus it either runs through at least six parallel strands of the diagram, or it intersects the diagram in the alternating tangle. Within the alternating tangle, the number of intersection points must be at least six. Hence in either case, the wave will intersect the diagram in at least six points. This finishes the claim for $H_1$.

The proof for $H_2$ is similar, only now cut $H_2$ into balls bounded by pants along the plane of projection. This cuts $F$ into two pants, one above and one below the plane of projection, as shown in \reffig{Represent}.
\begin{figure}
  \includegraphics{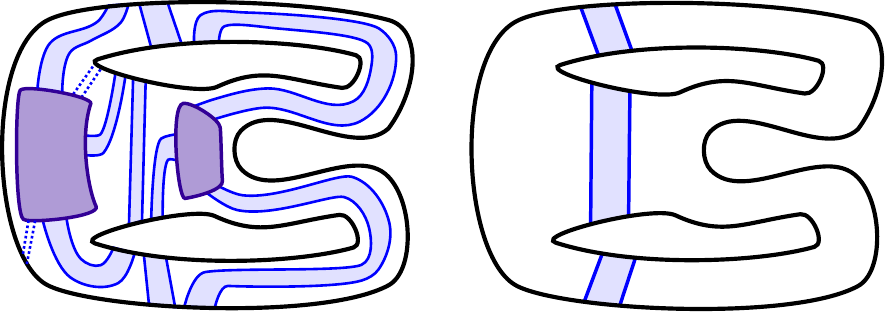}
  \caption{The form of the two pants obtained by cutting $F$ along the projection plane. The pants on the left lies above the plane of projection, and contains two alternating tangles plus a large number of parallel strands (at least six) following the arcs shown around those tangles (the dotted lines represent arcs in $K^{*}$ but not in $K$). On the right is the pants below the plane of projection, with at least six strands running in the shaded regions shown.}
  \label{Fig:Represent}
\end{figure}
Denote the boundary curves of the pants by $x$, $y$, and $z$. A compressing disc $E$ for $H_2$ can be isotoped to intersect $x$, $y$, and $z$ minimally, and to intersect the discs bounded by these curves only in arcs. Then an outermost arc of $E$ has boundary forming a wave. This wave separates boundary components, so again it either runs through the alternating tangle, meeting the knot diagram at least six times, or runs through a region with a high number of parallel strands, again meeting the knot diagram at least six times. 
\end{proof}

The final step in the construction of the full sequence of knots for \refthm{MainVolume} is to twist along pairs of curves. Fix an integer $m>0$. Add $m$ pairs of unknotted, unlinked components to the diagram $C_1, C_2, \dots, C_{2m-1}, C_{2m}$ as shown in \reffig{TwistCurves}. 

\begin{figure}
  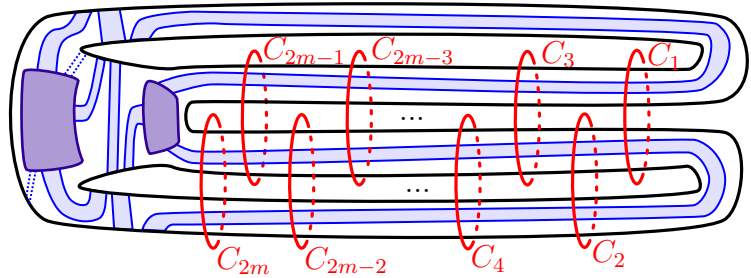
  \caption{The curves to twist along.}
  \label{Fig:TwistCurves}
\end{figure}

\begin{lemma}\label{Lem:Linking}
  Each curve $C_i$ as above has linking number $0$ with $K_T$ or $K_T^{*}$.
\end{lemma}

\begin{proof}
Consider the original template. Each curve $C_i$ bounds a disc in $S^3$ that is disjoint from all the curves of the original template, except meets each of the six parallel strands running over the handles of $H_1$ exactly twice. Orient these strands in the same direction, as in the construction of $K$ and $K^{*}$. Note each strand runs through the disc in one direction, then runs through again in the opposite direction, so each such strand has linking number zero with $C_i$. When we modified the template above to create the knot, in Subsection~\ref{Sec:SubConstruction}, we first inserted crossings at vertices. This did not affect orientation of strands through the $C_i$. We then added bigons to crossings on $F$, chosen in such a way that the orientation of the six strands around handles is not affected, as in \reffig{TwistOrientation}. Finally we twisted along $u$, $v$, and $w$. Again this had no effect on orientation of strands. Thus the linking number of $K_T$ or $K_T^{*}$ with any $C_i$ is zero.
\end{proof}

For each $C_i$, select an  integer $t_i\geq 7$, and perform $t_i$ full twists along $C_i$. That is, drill $C_i$ and perform $1/t_i$ Dehn filling. All choices of $m$ and all choices of integers $t_i\geq 7$ give a countable collection of knots in $S^3$. Denote the sequence of knots obtained from twisting $K_T$ by $\{K_n\}$, $n\in\NN$. Denote the sequence obtained from $K_T^{*}$ by $\{K_n^{*}\}$, $n\in\NN$. 

\begin{lemma}\label{Lem:WGA}
Any knot $K_T$, $K_T^{*}$, $K_n$, and $K_n^{*}$ is a weakly generalized alternating knot (WGA), with representativity at least six and edge-representativity at least four.
Furthermore, for every $n$,  the crossing number on $F$, $c_F(\pi(K_n))$, is bounded, equal to $c_F(\pi(K_T))=c_F(\pi(K))$, and $c_F(\pi(K_n^{*})) = c_F(\pi(K_T^{*}))=c_F(\pi(K^{*}))$. Similarly, the diagram is twist reduced, with bounded twist number on $F$.

Finally, each $K_n$ and $K_n^{*}$ has linking number zero with each $C_i$.
\end{lemma}

\begin{proof}
We have shown that $\pi(K_T)$ and $\pi(K_T^{*})$ are alternating on $F$, checkerboard coloured, weakly prime, and each component projects with at least one crossing by \reflem{KTalternating}. The representativity of each is at least six by \reflem{Representativity}.
Thus $\pi(K_T)$ and $\pi(K_T^{*})$ are weakly generalized alternating. They are twist reduced with edge-representativity at least four by \reflem{KTalternating}. We now show these properties for $\pi(K_n)$ and $\pi(K_n^*)$.

By choice, $H_1$ is constructed as a 3-ball $B$ with
with two unknotted and unlinked 2-handles. The co-cores of these 2-handles  are meridians of $H_1$. The cores can be joined to a wedge of two circles on which $H_1$ deformation retracts.
Twisting along the curves $C_i$ leaves the co-cores of the 2-handles of $H_1$ fixed and inserts a  4-string pure braid to the wedge of two circles on which $H_1$ retracts.
The complement $H_2$ is still a handlebody after twisting along the curves $C_i$. Indeed it is a complement of a 2-bridge link with an unknotting tunnel drilled out; see for  example \cite{Koba}. Thus after the twisting operation we still get a Heegaard splitting of $S^3$. We can isotope this Heegaard splitting to the original one
by sliding around the discs on $B$, on which the 2-handles of $H_1$
are attached, until the 4-string braid on the spine is trivialized. This implies that the effect of twisting along the curves $C_i$ is also achieved by a homeomorphism of $S^3$ that preserves $H_1$ and $H_2$.
The restriction of the homeomorphism on $F$ fixes a set of meridian curves of $H_1$ and restricts to a pure mapping class of the 4-holed sphere.

Since the property of bounding a meridian disc in $H_1$ (resp.\ $H_2$) is preserved under homeomorphisms of $H_1$ (resp.\ $H_2$) and intersection numbers of curves are preserved under homeomorphisms of $F$, it follows that $\pi(K_n)$ and $\pi(K_n^{*})$ have representativity at least six. Similarly, $\pi(K_n)$ and $\pi(K_n^{*})$ remain alternating on $F$ after twisting along $C_i$, with the same number of crossings on $F$ as before, and the same twist number, and they remain checkerboard colored, twist reduced, and weakly prime on $F$. 

Thus the representativity of $\pi(K_n)$ is identical to the representativity of $\pi(K)$, hence it is at least six. Similarly for $\pi(K_n^{*})$. Similarly, edge-representativity is identical in both cases, hence it is at least four. 
\end{proof}

\begin{proposition}\label{Prop:Hyperbolic}
Each WGA knot $K_T$, $K_T^{*}$, $K_n$, and $K_n^{*}$ has hyperbolic complement in $S^3$.
\end{proposition}

\begin{proof}
For $K_T^{*}$ and $K_n^{*}$, the result follows immediately from \refthm{HP}: These knots are WGA in $S^3$, all regions on $F-\pi(K_T^{*})$ and $F-\pi(K_n^{*})$ are discs, and the representativity is at least six, by Lemma~\ref{Lem:WGA}. Thus by that theorem, these knots are hyperbolic.

The result is not quite as straightforward for $K_T$ and $K_n$, because these knots have one region $E$ that is a 3-holed sphere, not a disc. For these knots, we show their complements are irreducible, boundary irreducible, anannular and atoroidal. For notational simplicity we will use $S^3- J$ to denote any of these manifolds.

We recall that ~\cite[Corollary~3.16]{HowiePurcell} states that if $\pi(L)$ is a reduced alternating link projection on a surface $F$ in a 3-manifold $Y$, them $Y-L$ is is irreducible and boundary irreducible.
Hence the claim that $S^3-J$ is irreducible and boundary irreducible follows from the fact that the knot is WGA.

By construction, $J$ has a WGA knot diagram $\pi(J)$ that has more than one twist region.
Thus part (2) of Theorem~\ref{Thm:HP2} implies that $S^3-J$  cannot contain an essential annulus.

Next we claim that $S^3-J$ is atoroidal. To prove it, suppose that $T$ is an essential torus in $S^3-J$. By part~(1) of \refthm{HP2}, the checkerboard surfaces of $\pi(J)$ cut $T$ into annuli $A_1, \dots, A_j$, such that each boundary component of each $A_i$ lies entirely in a single region of $F-\pi(J)$. Such a component cannot bound a disc (else use the fact that $T$ is incompressible and $S^3-J$ is irreducible to remove such an intersection), thus it must lie on a non-disc face of $F-J$. In our case there is only one non-disc face, namely the 3-holed sphere $E$. The three components of $\bdy E$ represent distinct free homotopy classes in $H_1$. Thus no two of these components can co-bound an annulus in $H_1$.

It follows that the components $\partial A_i$ must be parallel to a single component of $\partial E$ and parallel to each other. Then $A_i$ is an annulus in a handlebody with parallel boundary components on the boundary of the handlebody. Thus $A_i$ co-bounds a solid torus in the handlebody with an annulus on the boundary, and either $A_i$ can be isotoped away from $F$, or $J$ lies inside that solid torus. However by Lemma~\ref{Lem:WGA}, $J$ has edge-representativity at least four.
If $J$ were contained in a solid torus parallel to an annulus  on the boundary of $H_1$, then there would be an essential curve on $\bdy H_1$ disjoint from that annulus, implying the edge-representativity is zero. This is impossible. Thus $J$ cannot lie entirely within a solid torus defined by $A_i$, and so we may isotope $T$ to avoid the intersections of $\bdy A_i$. Repeating this argument, all the intersections $T\cap F$ can be eliminated by isotoping $T$. But then $T$ is inessential in $S^3-J$, a contradiction.
\end{proof}

%%%%%%%%%%%%%%%%%%%%%%%%%%%%%%%%%%%%%%%%%%%%%%%%%%%%%%%%%%%%%%%%%
\section{Proofs of main results}\label{Sec:Cusps}

This section completes the proofs of the volume result \refthm{MainVolume}, and its consequence for cusp densities, \refcor{Cusps}.
We will prove \refthm{MainVolume} by showing that the volumes of the knots $\{K_n\}$ or $\{K_n^{*}\}$ approach the volumes of another family $\{L_m\}$ with the property that $\vol(L_m)\to  \infty$ as $m\to \infty$.

Let $J$ be any knot in $\{K_n\}_{n\in \NN} \cup \{K^{*}_n\}_{n\in \NN}$ and let $L_m=J \cup (\bigcup_{i=1}^m C_{2i-1} \cup C_{2i})$, where  $C_1, C_2, \dots, C_{2m-1}, C_{2m}$  are the curves shown in \reffig{TwistCurves} that produce $J$. 
We will show that the links $L_m$ are hyperbolic. To do so, we use the following result of Gabai \cite[Corollary~2.4]{Gabai}. 

\begin{proposition}[Gabai]\label{Prop:Gabai}
Let $M$ be a Haken 3-manifold with toroidal boundary. If $S$ is a closed surface in $M$ that is not a boundary parallel torus, and such that $S$ is a Thurston norm minimizer in $H_2(M, \partial M)$,  then  $S$ remains norm minimizing in all but at most one of the 3-manifolds obtain by Dehn filling a along a single component of $\partial M$. 
In particular, $S$ remains incompressible  in all but at most one of the 3-manifolds obtain by Dehn filling along a single component of $\partial M$. 
\qed
\end{proposition}

Recall the definition of the Thurston norm. If $S$ is a connected closed surface, then define $\chi_{-}(S)$ to be the negative of the Euler characteristic of $S$, or $-\chi(S)$, if $S$ is not a 2-sphere. If $S$ is a 2-sphere, define $\chi_{-}(S)$ to be zero. If $S$ is not connected, then $\chi_{-}(S)$ is defined to be the sum of the complexities of the components of $S$. Given $S$ as above, let $c$ denote the class of $S$ in $H_2(M, \partial M)$. The \emph{Thurston norm} of $c$ is the minimum complexity $\chi_{-}(S)$
over all surfaces representing $c$. A representative $c=[S]$ that realises the Thurston norm of $c$ is called a \emph{Thurston norm minimizer} in its class.

\begin{proposition}\label{Prop:AugmentedHyperbolic}
The links $\{L_m\}_{m\in \NN}$ are  hyperbolic.
\end{proposition}

\begin{proof}
To see that $S^3-L_m$ is boundary irreducible, note that since $S^3-J$ is irreducible and boundary irreducible (by \refprop{Hyperbolic}), a boundary reducing disc $D$ in $S^3-L_m$ must have its boundary on one of the boundary components corresponding to some $C_i$.
Note that for $i$ odd, $C_i$ does not bound a disc in the complement of $J$. For $i$ even, $C_i$ bounds a disc in the complement of $J$, but this disc is pierced by $C_{i-1}$. Hence no such disc exists,
and $S^3-L_m$  is boundary irreducible.

Now suppose that $S^3-(J\cup C_1)$ is reducible. Then there is a sphere in $S^3-( J\cup C_1)$ that separates $J$ from $C_1$. This is impossible since $C_1$ is not homotopically trivial in $S^3-J$. Similarly suppose
$S^3-( J\cup C_1\cup C_2)$ is reducible. Then, since $S^3-( J \cup C_1)$ is irreducible, there must be a 2-sphere separating $C_2$ from $S^3-( J \cup C_1)$. This is again impossible since $C_2$  is not homotopically trivial in
$S^3-(J\cup C_1)$.
Continuing inductively we conclude that $S^3-L_m$ is irreducible.

Next we argue that  $S^3-L_m$ is atoroidal.
Suppose that there is an essential (i.e.\ incompressible and non-boundary parallel) torus $T_1$ in $M_1=S^3-(J\cup C_1)$. Consider the 3-manifolds  obtained from $M_1$ by $1/t_1$ Dehn filling along $C_1$. 
Each of these 
3-manifolds is the complement of a knot obtained from $J$ by $t_1$ full twists along the disc bounded by $C_1$. Hence for $t_1$ large enough, each of these 3-manifolds
is the complement of a WGA knot in one of the sequences $\{K_n\}$ or $\{K_n^{*}\}$; see discussion before the statement of \reflem{WGA}. Let us denote the corresponding knot by $J(t_1)$. 

Since $\chi(T_1)=0$, it is a Thurston norm minimizer in
$H_2(M_1, \partial M_1)$. Since $M_1$ is the complement of a link, it is Haken. Applying Gabai's \refprop{Gabai} to the essential torus $T_1\subset M_1$, we conclude $T_1$ remains incompressible in all the manifolds $S^3-J(t_1)$, for $t_1>>0$.  On the other hand $S^3-J(t_1)$ is the complement of a hyperbolic WGA knot by \refprop{Hyperbolic}, and as such it cannot contain an essential torus.
The only possibility is that $T_1$ is boundary parallel in $S^3-J(t_1)$. This implies that $C_1$ lies in a neighbourhood $\partial(N(J(t_1)))\times I$ of the torus boundary of the knot complement, and thus it can be homotoped to lie on  $\partial(N(J(t_1)))$.

By \reflem{WGA}, $C_1$ has linking number zero with $J(t_1)$. From the above discussion, it can be homotoped to lie on $\partial(N(J(t_1)))$.
Hence $C_1$ is either homotopically trivial on $\partial(N(J(t_1)))$ or homologous to a longitude of $\partial(N(J(t_1)))$. 
Suppose that $C_1$ is homotopically trivial on  $\partial(N(J(t_1)))$. Then $C_1$ can be made to bound a disc in a 3-ball that is disjoint from $N(J(t_1))$ and from $H_1$. But this contradicts our choice of $C_1$.
Next, suppose $C_1$ homologous to a longitude of $\partial(N(J(t_1)))$,  and thus
 a knot isotopic to $J(t_1)$ in $S^3$.
However this is impossible since $C_1$ bounds a disc in $S^3$ and $J(t_1)$ does not.
We conclude tha  $M_1=S^3-(J\cup C_1)$ is atoroidal.

Inductively, suppose $1<i\leq m$ and $(S^3-J)-(\bigcup_{j=1}^{i-1} C_j)-C_i$ contains an essential torus. Again consider $1/t_i$ Dehn fillings on $C_i$, yielding a manifold $(S^3-J(t_i))-(\bigcup_{j=1}^{i-1} C_j)$ where $J(t_i)$ is one of the WGA knots in $\{K_n\}$ or $\{K_n^{*}\}$. For $t_i>>0$, Gabai's \refprop{Gabai} again implies that the Dehn filling contains an incompressible torus. By induction, this must be boundary parallel in $(S^3-J(t_i))-(\bigcup_{j=1}^{i-1} C_j)$. Thus $C_1$ lies in a tubular neighbourhood of one of the link components $C_j$ or $J(t_i)$. By  linking number considerations as above, it cannot lie in a neighborhood of $J(t_i)$. By construction, it does not lie in a neighborhood of a component $C_j$. This contradiction proves the link is atoroidal. 

Finally note that since no components of $L_n$ co-bound an annulus in $N_n=S^3-L_n$, $N_n$ contains no essential annuli and hence  by Thurston's hyperbolization theorem $N_n$ is hyperbolic.
\end{proof}

We are now ready to prove the first main theorem, \refthm{MainVolume}.  
\begin{theorem}\label{Thm:MainVolume}
There exists a family of hyperbolic knots $\{K_n\}_{n\in \NN}$ in $S^3$ with $\vol(K_n)$ approaching infinity as $n\to\infty$, and such that $K_n$ satisfy the following properties.
\begin{itemize}
\item  $K_n$ has a weakly generalised alternating projection onto a Heegaard surface $F$ of genus two.
\item The representativity $r(\pi(K_n), F)$ is at least six.
\item The crossing number $c_F(\pi(K_n))=c$ on $F$ is the same for all $n$, and the twist number $t_F(\pi(K_n))=t$ on $F$ is the same for all $n$.
\item Furthermore, we can take the knots $K_n$ so that all the regions of the projections $\pi(K_n)$ on $F$ are discs, and the edge-representativity is strictly greater than two. Alternatively, we can take the $K_n$ to have a region that is not a disc.
\end{itemize}
In particular, neither the crossing number $c_F(\pi(K))$ nor the number of twist regions $t_F(\pi(K))$ on $F$ can give an upper bound on volume for a weakly generalised alternating link.
\end{theorem}

Observe that \refthm{MainVolume} immediately implies \refthm{MainVolumeIntro}.

\begin{proof}[Proof of \refthm{MainVolume}]
Fix $V>0$. We will show there exist infinitely many WGA knots with volume strictly greater than $V$. First, choose an integer $m>0$ such that
$2m\cdot v_3 > V$.

To produce the examples with disc regions and edge-representativity strictly greater than two, consider a knot $J$ in $\{K_n^{*}\}$ obtained by Dehn filling $2m$ cusps of the $2m+1$ cusped manifold
$M=(S^3-K_T^{*})-(\bigcup_{i=1}^m C_{2i-1} \cup C_{2i})$. These have generalised diagrams whose complementary regions on $F$ are all discs. We argue that infinitely many such $J$ have volume larger than $V$.
The link complement $M$ is hyperbolic by \refprop{AugmentedHyperbolic}. Because it has $2m+1$ cusps, by Adams~\cite[Theorem 2.5]{Adams1},
\[ \vol(M) \geq (2m + 1)\, v_3 > (2m)\,v_3 >V. \]

Recall that $J$ is obtained by filling all the cusps of $M$.
For  $i=1, \dots, 2m$, let $t_i$ be such that $J$ is obtained by $1/ t_i$ Dehn filling on the component $C_i$.
Thurston's hyperbolic Dehn filling theorem~\cite{thurston:notes} implies that as the Dehn filling coefficients $t_i$ approach infinity, the volume of $S^3-J$ approaches $\vol(M)$ from below. In particular, for infinitely many such knots, the volume of the knot complement is strictly greater than $2m\cdot v_3>V$, as required.

Finally, to produce examples with a complementary region on $F$ that is not a disc, run the same argument as above on knots in $\{K_n\}$. 
\end{proof}

This leads immediately to the proof of \refcor{Cusps}: that there exist weakly generalised hyperbolic knots in $S^3$ with arbitrarily small cusp density. 
 
\begin{proof}[Proof of \refcor{Cusps}]
Let $\{K_n^{*}\}_{n\in \NN}$ be the family of knots of \refthm{MainVolume} with disc regions. Let $J$ be any knot in this family.
By \refthm{HP}, the checkerboard surfaces $S_1$, $S_2$ corresponding to the WGA knot $\pi(J)$ are essential in $S^3-J$. Let $i(\bdy S_1, \bdy S_2)$ denote the minimal intersection number of $\bdy S_1, \bdy S_2$ on the torus boundary of the maximal cusp of $K^{*}_n$.
So $S_1, S_2$ are checkerboard surfaces of an alternating  knot projection on a Heegaard surface $F\subset S^3$
and such that all the regions of the projection are discs.  For such a situation the intersection number $i(\bdy S_1, \bdy S_2)$ and the quantity
$|\chi(S_1)| + |\chi(S_2)|$ are computed in the proof of \cite[Theorem~1.2]{BuKa}, where it its found that
\[ |\chi(S_1)| + |\chi(S_2)| = c_F(\pi(K_n^{*}))-\chi(F)=c_F(\pi(K_n^{*}))+2 \  \ {\rm and} \ \ 
i(\bdy S_1, \bdy S_2)=2c_F(\pi(K_n^{*})). \]
By \refthm{BK}, we obtain
\[ {\rm CV}(K_n^{*})\leq \dfrac{9}{2}\, c_F(\pi(K_n^{*}))\ \left(1+{{2}\over {c_F(\pi(K_n^{*})}}\right)^2. \]
By \reflem{WGA}, $c_F(\pi(K_n^{*}))=c_F(\pi(K^{*}))$, and hence crossing number on $F$ is bounded independently of $n$. Thus the cusp volume of $K_n^{*}$ is uniformly bounded from above.
On the other hand, by \refthm{MainVolume} we have $\vol(K_n^{*})\to \infty$ as $n\to \infty$.
Since the cusp density is by definition the quotient of ${\rm CV}(K_n^{*})$ by  $\vol(K_n^{*})$, the result follows.
\end{proof}

%%%%%%%%%%%%%%%%%%%%%%%%%%%%%%%%%%%%%%%%%%%%%%%%%%%%%%%%%%%%%%%%%
\section{Links in thickened surfaces}\label{Sec:Thickened}

The goal of this section is to prove \refthm{SxIUpperBound} and \refcor{TorAltUpperBound}.

\begin{proof}[Proof of \refthm{SxIUpperBound}]
Let $Y=S\times[-1,1]$, $F=S\times\{0\}$, and $K$ a WGA link as in the statement of the theorem. 
The fact that $Y-K$ is hyperbolic follows from \refthm{HP}: $Y-F$ is atoroidal, all regions are discs by hypothesis, and $F$ is incompressible, so the representativity is infinite. Thus the hypotheses of \refthm{HP} hold and the link is hyperbolic. Recall that when the genus of $S$ is at least two, there are infinitely many hyperbolic structures, due to Thurston~\cite{thurston:bulletin}. We choose the hyperbolic structure to be such that the higher genus boundary components are totally geodesic, as in \cite{AdamsCalderonMayer, HowiePurcell}; we refer to those papers for further details. 

The lower volume bounds are given by \refthm{HP}: note $\bdy Y$ consists of two copies of $F$, thus the terms $-\chi(F)-\chi(\bdy Y)$ become $-3\chi(F)$. When $F$ is a torus, this is zero.
  
We obtain upper volume bounds similarly to \cite[Appendix]{lackenby:alt-volume}, by forming a fully augmented link; see also \cite{Purcellsurvey}. In particular, augment each twist region on $F$ by encircling it with a crossing circle, which bounds a 2-punctured disc punctured by $K$. Untwist along each crossing circle: remove all pairs of crossings from each twist region. This is done by a homeomorphism of the complement of $K$ and the crossing circles. Let $L$ denote the resulting link. It consists of crossing circle components, bounding 3-punctured spheres meeting $F$ transversely, and components that came from $K$. Components from $K$ are embedded on $F$, except possibly in a neighbourhood of each crossing circle, where they may have a single crossing. 

As in \cite[Appendix]{lackenby:alt-volume}, the fully augmented link on $F$ in $S\times[-1,1]$ can be decomposed into two geometric pieces, with totally geodesic boundary. The decomposition is obtained by cutting along two types of surfaces. The first surface consists of all the 2-punctured discs bounded by crossing circles. These are totally geodesic~\cite{Adams:Thrice}. Shade them. They divide $F-L$ into regions that will form the second surface. As in the case of a fully augmented link in $S^3$, a reflection through the projection surface takes $Y-L$ to a fully augmented link. The reflection reverses each single crossing, but following the reflection by a full-twist homeomorphism about the corresponding crossing circle takes the augmented link diagram back to itself. Thus there is an orientation-reversing self-homeomorphism of $Y-L$. This fixes pointwise each of the regions of $F-L$ away from crossing discs. It follows from Mostow--Prasad rigidity that these white regions glue to a white surface that is totally geodesic in $Y-L$. 

Cut along the white surface and the shaded discs. These give a decomposition of $Y-L$ into two isometric pieces, each homeomorpic to $S\times[0,1]$, with $S\times\{0\}$ marked by white and shaded faces meeting at ideal vertices. The shaded faces, coming from 2-punctured discs, are triangular. The white faces come from regions of $F-L$. White and shaded faces meet at right angles. 

Consider first the case that $S=T^2$, a torus. Then $T^2\times\{1\}$ is realised as a toroidal cusp of $Y-L$. 
Cone each face on $T^2\times\{0\}$ to the cusp $T^2\times\{1\}$ to produce an ideal pyramid. The cone over a shaded face is a tetrahedron, one for each shaded face on each copy of $T^2\times[0,1]$. Two pyramids over a white face glue together to form a bipyramid; perform stellar subdivision of each white bipyramid. That is, subdivide a bipyramid into tetrahedra by removing the white face, taking an infinite geodesic dual to that face, and adding in triangles to divide the region into tetrahedra. If the white face has $d$ edges, the resulting subdivision gives $d$ tetrahedra; this subdivision is exactly as in \cite[Appendix]{lackenby:alt-volume}. 

The count is also identical to that paper: Each crossing circle becomes four shaded triangles, giving rise to four tetrahedra. Each edge on $T^2\times[0,1]$ that borders a white face gives rise to one tetrahedron, for both copies of $T^2\times[0,1]$. The subdivision gives three edges per crossing circle, each appearing twice on $T^2 \times[0,1]$. This gives a total of six additional tetrahedra per crossing circle.

Thus we may subdivide the augmented link on the torus into $10\, t_F(\pi(K))$ tetrahedra. The maximum volume of a tetrahedron is $\vtet = 1.01494\dots$. Thus
\[ \vol(Y-L) \leq 10\,\vtet\cdot t_F(\pi(K)). \]

For the case $Y=S\times[-1,1]$ with $S$ of higher genus, the argument is similar, but rather than coning to an ideal vertex at $S\times\{1\}$, we cone white and shaded faces to \emph{ultra-ideal} vertices coming from the totally geodesic surfaces at $S\times\{1\}$, as in \cite{AdamsCalderonMayer}. These are also referred to as hyperideal vertices, or truncated vertices, in the literature. Briefly, viewing $\HH^3$ in the Klein model, an ultra-ideal vertex lies outside the boundary at infinity. It defines a unique circle on the boundary at infinity, bounding a totally geodesic surface in $\HH^3$, and we cut this off to form the truncated, ultra-ideal vertex. A \emph{generalised tetrahedron} has at least one ultra-ideal vertex; topologically the tetrahedron is truncated at each ultra-ideal vertex. Geometrically, the totally geodesic faces of the tetrahedron meet a totally geodesic truncation face at right angles. 
See also \cite{Ushijima} for a more detailed discussion.

In our situation, in the case of a white face, coning to an ultra-ideal vertex gives a truncated pyramid over the white face as a base. Gluing two white faces together produces a truncated bipyramid, with two ultra-ideal vertices, and an ideal $d$-gon corresponding to the white face at the centre. As above, stellar subdivision divides these into generalised tetrahedra with two ideal vertices and two ultra-ideal vertices. 

For each shaded face, take a truncated tetrahedron, with three ideal vertices and one ultra-ideal vertex. For ease of stating the bound, glue two shaded faces into a bipyramid over a triangle with two ultra-ideal vertices, and then perform stellar subdivision. Now each pair of shaded faces is divided into three generalised tetrahedra with two ideal vertices and two ultra-ideal vertices.

We count the number of generalised tetrahedra. Each crossing circle becomes four shaded triangles, giving rise to six generalised tetrahedra after gluing them in pairs and performing stellar subdivision. Each white face is bordered by some number of edges; each such edge gives rise to one generalised tetrahedron after stellation. Our decomposition into white and shaded faces above gives rise to three edges per crossing circle, each appearing twice as an edge of a white face. This gives a total of six additional generalised tetrahedra per crossing circle, or twelve total per crossing circle. (Observe that the count of tetrahedra meeting a white face is identical to the count above in the case of the torus, which is identical to the count in \cite[Appendix]{lackenby:alt-volume}. The only difference is that now we count generalised tetrahedra.)

By Adams, Calderon, and Mayer~\cite[Corollary~3.4]{AdamsCalderonMayer}, the maximum volume of a generalised tetrahedron with two ideal vertices and two ultra-ideal vertices is $\voct/2$, where $\voct=3.66386\dots$ is the volume of a regular ideal octahedron. There are $t_F(\pi(K))$ crossing circles. Thus
\[ \vol(Y-L) \leq \frac{\voct}{2} \cdot 12\, t_F(\pi(K)) = 6\voct\, t_F(\pi(K)).\]

The link $K$ is obtained from $L$ by Dehn filling, and volume strictly decreases under Dehn filling~\cite{thurston:notes}. Thus $\vol(Y-K) <\vol(Y-L)$, giving the result. 
\end{proof}

Theorem~\ref{Thm:MainVolume} required a generalised projection surface of genus two. A similar construction could likely be made for a higher genus surface to give unbounded volume. However, the same result will not hold for a projection surface that is a torus. In particular, we now argue that there \emph{is} an upper bound on volume in terms of $t_F(\pi(K))$ for a WGA knot $K$ on a Heegaard torus in $S^3$, or indeed in any lens space.

\begin{corollary}\label{Cor:TorAltUpperBound}
Suppose $K$ is a link that has a weakly generalised alternating projection $\pi(K)$ to a Heegaard torus $F$ in $Y=S^3$, or in $Y=L(p,q)$ a lens space. Suppose that $\pi(K)$ is twist reduced, the regions of $F-\pi(K)$ are discs, and the representativity satisfies $r(\pi(K),F)>4$. Then $Y-K$ is hyperbolic, and
\[  {\voct\over 2}\  t_F(\pi(K)) \leq \vol(Y-K) < 10\,\vtet\cdot t_F(\pi(K)).\]
\end{corollary}

\begin{proof}%[Proof of \refcor{TorAltUpperBound}]
Set $Y$ to be the manifold $S^3$ or $L(p,q)$. By \refthm{HP}, $Y-K$ is hyperbolic. 

The Heegaard torus $F$ bounds solid tori $V_1$ and $V_2$ on either side. The core curves of the solid tori form a Hopf link in $S^3$ if $Y=S^3$, and some 2-component link if $Y=L(p,q)$. In the case of the Hopf link in $S^3$, drill the Hopf link from $S^3-K$. The complement of the Hopf link in $S^3$ is the manifold manifold $T^2\times \RR$. The surface $F$ is embedded as $T^2\times \{0\}$, and the knot $K$ lies on $F$ as a projection surface. Similarly, if $Y$ is a lens space, then drilling the two core curves of the solid tori of $Y$ gives the manifold $T^2\times \RR$ with the surface $F$ embedded as $T^2\times\{0\}$, with the knot $K$ on $F$. 

Let $X$ be $T^2\times \RR$ with (open) horoball neighbourhoods of the cusps at $\pm \infty$ removed; thus $X$ is homeomorphic to $T^2\times[-1,1]$. The knot $K$ becomes a WGA knot on $T^2\times\{0\}$ inside $X$. 

Now apply \refthm{SxIUpperBound}. That theorem implies that $\vol(X-K) < 10\,\vtet\cdot t_F(\pi(K)).$

Finally, the complement of the WGA knot $Y-K$ is obtained from $X-K$ by Dehn filling along the 2-component link at the cores of the solid tori. Since volume strictly decreases under Dehn filling~\cite{thurston:notes}, $\vol(Y-K) < 10\,\vtet\, t_F(\pi(K))$.
\end{proof}

\begin{remark}
The proof of the upper bound of \refcor{TorAltUpperBound} uses restrictions on representativity only to guarantee the link complement is hyperbolic, using \refthm{HP}. In \cite{Adams:ToroidallyAlt}, Adams finds hyperbolicity conditions for alternating links on Heegaard tori in $S^3$ and in lens spaces, which he calls toroidally alternating. Our upper volume bound will hold for toroidally alternating links that are hyperbolic, with no change to the proof. 
\end{remark}

\begin{remark}
Dasbach and Lin \cite{DL} proved that the twist number of a reduced, twist reduced alternating link diagram on $S^2$ is an invariant of the link, by relating the twist number to coefficients of the Jones polynomial.
In \cite{BaKa}, Bavier and Kalfagianni showed that the same is true for the links of \refthm{SxIUpperBound} under the additional hypothesis that the edge-representativity of $\pi(K)$ is at least four. 
\end{remark}

%%%%%%%%%%%%%%%%%%%%%%%%%%%%%%%%%%%%%%%%%%%%%%%%%%%%%%%%%%%%%%%%%
\section{Further remarks}
In this section we collect a few remarks, questions, and observations related to the main results. 

\subsection{A proof via generalised augmented links}

Theorem~\ref{Thm:MainVolume} was proved by finding very straightforward lower bounds on volume using the number of cusps. In fact, better lower bounds on volume can be obtained for the knots $\{K_n\}$ with a 3-holed sphere region in the diagram. This is done by producing a generalised fully augmented link from the knot $K_n$. That is, in addition to adding components $C_1, \dots, C_{2m}$, add components $u,v,w$ from the construction as in \reffig{RedCurves}, and for each twist region on $F$, add a crossing circle encircling that twist region, bounding a 3-punctured sphere. Generalised fully augmented links were considered in \cite{PurcellSlopeLengths, Purcell:multiplytwisted, PurcellSeifert}. 

Similarly to the proof of \refprop{AugmentedHyperbolic}, it can be shown that the fully augmented link is hyperbolic by ruling out essential spheres, discs, tori, and annuli; the arguments follow fairly similarly to those of \cite{PurcellSeifert}. Assuming hyperbolicity, \cite[Theorem~4.2]{Purcell:multiplytwisted} implies the volume satisfies
\[ \vol(K) > 0.64756 \cdot (75 + 2m-1). \]
Here $m$ is the integer such that $K$ is obtained from $K_T$ by twisting at least seven times along curves $C_1, \dots, C_{2m}$. The constant $0.64756$ can be improved as well as the amount of twisting increases, using a
a theorem of  Futer Kalfagianni and Purcell~\cite{fkp:filling}. Indeed, this was our first approach to \refthm{MainVolume}.

\subsection{Lower bounds}
The lower volume bounds from Howie and Purcell \cite{HowiePurcell} on volumes of WGA knots in terms of twist regions on $F$ only apply to WGA knots with all complementary regions in $F-\pi(K)$ homeomorphic to discs. This does not apply to the knots $\{K_n\}$ with a 3-holed sphere region. It leads to the following natural question. 

\begin{question}
  Suppose $K$ admits a weakly generalised alternating projection $\pi(K)$ onto a generalised projection surface $F$ in $S^3$, or more generally in any compact orientable 3-manifold $Y$. Suppose the representativity is at least six, but at least one complementary region of $\pi(K)$ on $F$ is not a disc. Suppose also that $K$ is hyperbolic. Is there a lower bound on the volume of the complement of $K$ in terms of the number of twist regions of $K$ on $F$?
\end{question}

\subsection{Turaev surfaces and adequate knots}
In closing, we mention that all knots in $S^3$ admit alternating checkerboard projections on certain Heegaard surfaces called \emph{Turaev surfaces}; see \cite{DFKLS}.
In this setting an additional diagrammatic condition, called adequacy, also produces knots whose geometry shares many common properties with those of the usual alternating knots, such as volume bounds~\cite{fkp:guts}, topology and geometry of checkerboard surfaces~\cite{Ozawa, fkp:TAMS}, and hyperbolicity~\cite{fkp:CAG}. However, the techniques of WGA links typically do not apply, because such diagrams have low representativity.

\bibliographystyle{amsplain}
\bibliography{biblio.bib}
\end{document}